\documentclass[final,12pt,3p,pslatex]{article}%
\usepackage{amsfonts}
\usepackage{amssymb}
\usepackage{amsthm}
\usepackage{amsmath}
\usepackage{graphicx}%
\providecommand{\U}[1]{\protect\rule{.1in}{.1in}}
\newtheorem{theorem}{Theorem}

\newtheorem{lemma}{Lemma}

\newtheorem{proposition}{Proposition}

\newcommand{\E}{{\mathbb{E}}}

\begin{document}

\title{Analysis of a Splitting Estimator for Rare Event Probabilities in Jackson Networks}
\author{Jose Blanchet, Kevin Leder and Yixi Shi}
\maketitle

\begin{abstract}
We consider a standard splitting algorithm for the rare-event simulation of
overflow probabilities in any subset of stations in a Jackson network at level
$n$, starting at a fixed initial position. It was shown in \cite{DeanDup09}
that a subsolution to the Isaacs equation guarantees that a subexponential
number of function evaluations (in $n$) suffice to estimate such overflow
probabilities within a given relative accuracy. Our analysis here shows that
in fact $O\left(  n^{2\beta+1}\right)  $ function evaluations suffice to
achieve a given relative precision, where $\beta$ is the number of bottleneck
stations in the network. This is the first rigorous analysis that allows to
favorably compare splitting against directly computing the overflow
probability of interest, which can be evaluated by solving a linear system of
equations with $O(n^{d})$ variables.

\end{abstract}





\section{Introduction\label{intro}}

The development of rare-event simulation algorithms for overflow probabilities
in stable open Jackson networks has been the subject of a substantial amount
of papers in the literature during the last decades (see Section 2 for the
specification of an open Jackson network). A couple of early references on the
subject are \cite{PARWAL89} and \cite{ANAHEITSO90}. Subsequent work which has
also been very influential in the development of efficient algorithms for
overflows of Jackson networks include \cite{Villen94},
\cite{glasserman99multilevel}, \cite{GLAKOU95}, \cite{KroeseNicola},
\cite{JunejaNicola05}, \cite{DUPSEZWAN07}, \cite{NicolaZ07}, \cite{DupWang08}
and \cite{DeanDup09}. The survey papers of \cite{JUNSHA06} and \cite{BMRT09}
provide additional references on this topic.

The two most popular approaches that are applied to the construction of
efficient rare-event simulation algorithms are importance sampling and
splitting (see \cite{ASGLYNN}). Importance sampling involves simulating the
system under consideration (in our case the Jackson network) according to a
different set of probabilities in order to induce the occurrence of the rare
event. Then, one attaches a weight to each simulation corresponding to the
likelihood ratio of the observed outcome relative to the nominal / original
distribution. In splitting, on the other hand, there is no attempt to bias the
behavior of the system. Instead, the idea is to decompose the occurrence of
the rare event of interest (in our case overflow in a Jackson network) into a
sequence of nested \textquotedblleft milestone\textquotedblright\ events whose
subsequent occurrence is not rare. The rare event occurs when the last of the
milestone events occurs. The idea is to keep splitting the particles as they
reach subsequent milestones. Of course, each particle is attached a weight
corresponding to the total number of times it has split so that the overall
estimation (which is the sum of the weights corresponding to the particles
that make it to the last milestone) provides an unbiased estimator of the
probability of interest.

The most popular performance measure for efficiency analysis of rare-event
simulation algorithms for Jackson networks corresponds to that of
\textquotedblleft asymptotic optimality\textquotedblright\ or
\textquotedblleft weak efficiency\textquotedblright. In order to both explain
the computational complexity implied by this notion and to put in perspective
our contributions let us discuss the class of problems we are interested in.
Starting from any fixed state, we consider the problem of computing the
probability that the total number of customers in any fixed set of stations in
the network reaches level $n$ prior to reaching the origin. That is, the
probability that the sum of the queue lengths in any given set of stations
reaches level $n$ within a busy period. The number of stations in the whole
network is assumed to be $d$ and the number of bottleneck stations (i.e.
stations with the maximum traffic intensity in equilibrium) is $\beta$.

Weak efficiency guarantees that a subexponential number of replications (as a
function of the overflow level, say $n$) suffices to compute the underlying
overflow probability of interest within a given relative accuracy. In
contrast, as we shall explain in Section 2, overflow probabilities in the
setting of Jackson networks can be computed by solving a linear system of
equations with $O(n^{d})\footnote{Given two non-negative functions $f\left(
\cdot\right)  $ and $g\left(  \cdot\right)  $, we say $f\left(  n\right)
=O\left(  g\left(  n\right)  \right)  $ if there exists $c,n_{0}\in(0,\infty)$
such that $f\left(  n\right)  \leq cg\left(  n\right)  $ all $n\geq n_{0}$.}$
unknowns. It is well known that Gaussian elimination then takes $O\left(
n^{3d}\right)  $ operations to find an exact solution to such linear system.
Moreover, since in our case the associated linear system has some sparsity
properties the linear equations can be solved in as many as $O\left(
n^{3d-2}\right)  $ operations (see the discussion in Section 2). Our analysis
for the solution of the associated linear system of equations is not intended
to be exhaustive. Our objective is simply to make the point that naive Monte
Carlo (which indeed takes an exponential number of replication in $n$ to
achieve a given relative accuracy) is not the natural benchmark that one
should be using in order to test the performance of an efficient simulation
estimator for overflows in Jackson networks. Rather, a more natural benchmark
is the application of a straightforward method for solving the associated
system of linear equations. It would be interesting to provide a detailed
study of various methods for solving linear system of equations (such as
multigrid procedures) that are suitable for our environment and even combine
them with the ideas behind efficient simulation procedures. This, however,
would be the subject of an entire paper and therefore is left as a topic for
future research.

Our goal here is to analyze a class of splitting algorithms similar to those
introduced in \cite{Villen94} for the evaluation of overflow probabilities at
level $n$. Further analysis was given in \cite{DeanDup09}, where the authors
provide necessary and sufficient conditions for the design of the
\textquotedblleft milestone events\textquotedblright\ in order to achieve
subexponential complexity in $n$. Our contribution is to show that if the
milestone events are properly placed as suggested by \cite{DeanDup09}, the
splitting algorithm requires $O\left(  n^{2\beta+1}\right)  $ function
evaluations to achieve a fixed relative error. Since clearly the number of
bottleneck stations $\beta$ is at most $d$, the complexity of splitting is
$O\left(  n^{2d+1}\right)  $, which is substantially smaller than that of the
direct solution of the associated linear system. Our analysis therefore
provides theoretical justification for the superior performance observed when
applying splitting algorithms compared to directly solving the associated
linear system.

We believe that our results shed light into the type of performance that can
be expected when applying particle algorithms beyond the setting of Jackson
networks. This feature should be emphasized, specially given the fact that a
linear time algorithm for computing overflows in Jackson networks has been
developed very recently (see \cite{Blanchet09}). Contrary to particle methods,
which are versatile and that can in principle be applied in great generality,
the algorithm in \cite{Blanchet09} takes advantage of certain properties of
Jackson networks which are not shared by all classes of systems.

In addition, our results also provide interesting connections to recent
performance analysis studied in the context of state-dependent importance
sampling algorithms for a class of Jackson networks. These connections might
eventually help guide the users of rare event simulation algorithms decide
when to apply importance sampling or splitting. For instance, consider the
overflow at level $n$ of the total population of a tandem network with $d$
stations. The work of \cite{DUPSEZWAN07} proposes an importance sampling
estimator based on the subsolution of an associated Isaacs equation. In
particular, \cite{DUPSEZWAN07} shows that if exponential tiltings are applied
using the gradient of the associated subsolution as the tilting parameter
(depending on the current state), the corresponding algorithm is weakly
efficient. Turns out that there are many subsolutions that can be constructed
varying certain so-called \textquotedblleft mollification
parameters\textquotedblright. A recent analysis based on Lyapunov inequalities
given in \cite{BlanLedGlyn09} shows that a natural selection of mollification
parameters guarantees $O\left(  n^{2(d-\beta)+1}\right)  $ function
evaluations to achieve a given relative error. Our analysis here therefore
guarantees that one can achieve a running time of order $O\left(
n^{d+1}\right)  $ if one chooses importance sampling when there are more than
$d/2$ bottleneck stations in the network and splitting if there are less than
$d/2$ bottleneck stations. Although our analysis is still not sharp we believe
that our results provide a significant step in order to understand the
connections between splitting and importance sampling.

The rest of the paper is organized as follows. A brief discussion on
complexity and efficiency considerations is given in Section \ref{complexity}.
Then we discuss the necessary large deviations asymptotics for Jackson
networks required for our analysis in Section \ref{dynamics}. The introduction
of the splitting algorithm as well as connections to the theory developed in
\cite{DeanDup09} is given in Section \ref{splitting}. Our complexity analysis
is finally given in Section \ref{analysis}.

\section{Complexity and Efficiency}

\label{complexity} We shall review concepts of efficiency and complexity in
rare event simulation. We start our discussion in the context of a generic
class of rare event simulation problems. Consider a sequence of events
$\{E_{n},n=1,2...\}$ with $p_{n}\triangleq{\mathbb{P}}\left(  E_{n}\right)
\rightarrow0$ as $n\rightarrow\infty$ (Without loss of generality, we might
assume that $p_{n}\rightarrow0$ exponentially fast as $n\nearrow0$.) The
design of an efficient rare-event simulation algorithm is typically associated
with the construction of an unbiased estimator, say $\hat{p}_{n}$, such that
$p_{n}=\mathbb{E}\left[  \hat{p}_{n}\right]  $. A number of $m$ i.i.d.
replications $\{\hat{p}_{n}^{\left(  1\right)  },...,\hat{p}_{n}^{\left(
m\right)  }\}$ is produced, the average of which forms an estimate of $p_{n}$,
namely
\[
\hat{p}_{n}\left(  m\right)  =\frac{1}{m}\sum_{j=1}^{m}\hat{p}_{n}^{\left(
j\right)  }.
\]
By virtue of Chebyshev's inequality we obtain the following property for the
relative error, $|\hat{p}_{n}\left(  m\right)  -p_{n}|/p_{n}$, of the
estimate
\begin{equation}
{\mathbb{P}}\left(  |\hat{p}_{n}\left(  m\right)  -p_{n}|/p_{n}>\epsilon
\right)  \leq\frac{\text{Var}\left(  \hat{p}_{n}\right)  }{mp_{n}^{2}%
\epsilon^{2}}. \label{eqn:rel error bdd}%
\end{equation}
Hence, for a pre-determined upper bound $\epsilon$ of relative error, if we
choose the number of replications $m$ such that
\begin{equation}
m\geq\epsilon^{-2}\delta^{-1}\left(  cv_{n}\right)  ^{2},
\label{eqn: num of replication}%
\end{equation}
where $cv_{n}^{2}=\text{Var}\left(  \hat{p}_{n}\right)  /p_{n}^{2}$ is the
squared coefficient of variation of $\hat{p_{n}}$, we can guarantee that the
relative error is no larger than $\epsilon$ with probability at least
$1-\delta$.

Equation (\ref{eqn: num of replication}) stipulates that $m$ needs to grow at
least at the same rate as $cv_{n}^{2}$ does in order to keep the relative
error within a desirable threshold. If $cv_{n}^{2}$ grows at a subexponential
rate (i.e. if $\log\left(  cv_{n}\right)  ^{2}=o\left(  \log p_{n}\right)
,\,\text{as }n\nearrow\infty$) the estimator is said to be \emph{asymptotic
optimality, logarithmical efficient} or \emph{weakly efficient}. In this case,
the number of replication needs to increase subexponentially in $n$ to achieve
a prescribed level of relative accuracy. The name \textquotedblleft asymptotic
optimality\textquotedblright\ is derived from the fact that weak efficiency
implies that the exponential rate of decay to zero of the $\mathbb{E}\hat
{p}_{n}^{2}$ coincides with that of $p_{n}^{2}$ and therefore is maximal (by
virtue of Jensen's inequality).\newline\indent Obviously, one has to keep in
mind that weak efficiency measures the optimality of the estimator for a given
level of computational budget. Coming to the splitting algorithm, it is
apparent that the computational effort varies drastically with the degree of
splitting performed; one must therefore take into account the cost involved in
generating each replication of $\hat{p}_{n}$. We measure such cost in terms of
the number of elementary function evaluation which we will take to be simple
addition, multiplication, comparison, and the generation of a single uniform
random variable. When we incorporate the computational cost per replication of
the estimator, (\ref{eqn: num of replication}) says that the total number of
function evaluations needed has to keep pace with the work-normalized squared
coefficient of variation, i.e., $cv_{n}^{2}\cdot N_{n}$, where $N_{n}$ is the
cost per replication of $\hat{p}_{n}$. We will show in section \ref{analysis}
that $N_{n}$ is closely related to the expected total number of the survival
particles in a single run of the Splitting algorithm. \newline\indent Coming
back to the setting of Jackson networks. It is important to recognize that
overflow probabilities in such a setting can be obtained by solving a system
of linear equations. So, a reasonable benchmark for any simulation based
algorithm to be regarded \textquotedblleft efficient\textquotedblright\ is
indeed how fast one can solve such a system of linear equation by a direct
procedure. Jackson networks are basically multidimensional simple random walks
with constrained behavior on the boundaries. In particular, they are Markov
chains living on a countable state-space. The overflow probabilities can be
conveniently expressed as first passage time probabilities, which in turn can
be characterized as the solution to certain linear system of equations thanks
to its countable state-space Markov chain structure. We shall quickly review
how to obtain such linear system for a generic Markov chain $Q=\{Q_{k}%
:k\geq0\}$ living on a countable state-space $\mathcal{S}$ with transition
matrix $\{K\left(  x,y\right)  :x,y\in\mathcal{S}\}$. Let $A,\,B$ be two
disjoint subsets of $\mathcal{S}$, define $T_{A}\triangleq\inf\{k\geq0:X\in
A\}$, $T_{B}\triangleq\inf\{k\geq0:X\in B\}$ and put $p\left(  x\right)
={\mathbb{P}}_{x}\left(  T_{A}\leq T_{B}\right)  $. A simple conditioning
argument on the first transition leads to
\begin{equation}
p\left(  x\right)  =\sum_{y\in\mathcal{S}}K\left(  x,y\right)  p\left(
y\right)  \label{eqn:system of eqn}%
\end{equation}
subject to the boundary conditions
\[
p\left(  x\right)  =1\text{ for }x\in A,\quad p\left(  x\right)  =0\text{ for
}x\in B.
\]
In fact, $p\left(  \cdot\right)  $ is the minimum non-negative solution to the
above system (see \cite{BG07AAP}).

Now, if $Q$ describes the state of the embedded discrete time Markov chain
corresponding to a Jackson network with $d$ stations then $\mathcal{S}%
=\mathcal{Z}_{+}^{d}$. The transition dynamics of a Jackson network are
specified as follows (see \cite{RobertQN} p. 92). Inter-arrival times and the
service times are all independent and exponentially distributed random
variables. The arrival rates are given by the vector $\lambda=\left(
\lambda_{1},...,\lambda_{d}\right)  ^{T}$ and service rates are given by
$\mu=\left(  \mu_{1},...,\mu_{d}\right)  ^{T}$. (By convention all of the
vectors in this paper are taken to be column vectors and $^{T}$ denotes
transposition.) A job that leaves station $i$ joins station $j$ with
probability $P_{i,j}$ and it leaves the system with probability%
\[
P_{i,0}\triangleq1-\sum_{j=1}^{d}P_{i,j} .
\]
The matrix $P=\{P_{i,j}:1\leq i,j\leq d\}$ is called the routing matrix. We
shall consider open Jackson networks, which satisfy the following conditions:

\bigskip

\begin{enumerate}
\item[i)] $\forall i$, either ${\lambda}_{i}>0$ or ${\lambda}_{j_{1}}%
P_{j_{1}j_{2}}...P_{j_{k}i}>0$ for some $j_{1},...,j_{k}$.

\item[ii)] $\forall i$, either $P_{i0}>0$ or $P_{ij_{1}}P_{j_{1}j_{2}%
}...P_{j_{k}0}>0$ for some $j_{1},...,j_{k}$.

\item[iii)] The network is stable (i.e. a stationary distribution exists).
\end{enumerate}

\bigskip

These conditions simply require that each station will receive jobs either
directly from the outside or routed from other stations, and each job will
leave the system eventually. Our main interest lies in the evaluation of
$p_{n}\left(  x\right)  $ assuming that $B=\{0\}$ and $A_{n}=\{y:v^{T}y=n\}$
where $v$ is a binary vector which encodes a particular subset of the network
(i.e., the $i$-th position of the vector $v$ is $1$ if station $i$ falls in
the subset of interest, and $0$ otherwise). We shall denote by $V\left(
x\right)  =x^{T}v$ the mapping recording the total population in the stations
corresponding to the vector $v$. The case in which $v=\mathbf{1}=\left(
1,1,...,1\right)  ^{T}$ corresponds to the total population of the system. So,
$p_{n}\left(  x\right)  $, or more precisely $p_{n}^{V}\left(  x\right)  $,
corresponds to the overflow probability in the subset encoded by $v$ within a
busy period and starting from $x$ and that. In this setting, it follows (as we
shall review in the next section) that $p_{n}^{V}\left(  x\right)
\longrightarrow0$ exponentially fast in $n$ as $n\nearrow\infty$ and the
system of equations (\ref{eqn:system of eqn}) has $O\left(  n^{d}\right)  $
unknowns. Gaussian elimination requires $O\left(  n^{3d}\right)  $ function
evaluations to find the solution of such system. But since each state of the
Markov chain in this case has possible interactions with only a small fraction
of the entire state-space, it is therefore possible to permutate the states
(say in lexicographic order) so that the system is banded (i.e. the associated
matrix is sparse in the sense that its non-zero entries fall to a diagonal
band.) One can show that the bandwidth is $O\left(  n^{d-1}\right)  $, and
therefore solving such a banded linear system requires $O\left(  n^{d}%
\cdot\left(  n^{d-1}\right)  ^{2}\right)  =O\left(  n^{3d-2}\right)  $
operations (see, e.g., \cite{AG94}). \newline\indent Estimators that possess
weak efficiency (in a work-normalized sense) are guaranteed to run at
subexponential complexity. When comparing to the above \emph{polynomial}
algorithms of solving systems of linear equations, the efficiency analysis of
such estimators appears to be insufficient. We will show in later analysis
that the multilevel Splitting algorithm suggested by Dean and Dupuis
\cite{DeanDup09}, applied to estimate the overflow probabilities in Jackson
networks, requires fewer function evaluations than directly solving the
associated system of linear equations.

\section{Jackson Networks: Notation and Properties}

\label{dynamics} As we mentioned in the previous section, a Jackson network is
encoded by two vectors of arrival and service rates, ${\lambda}=\left(
{\lambda}_{1},...,{\lambda}_{d}\right)  ^{T}$ and $\mu=\left(  \mu_{1}%
,...,\mu_{d}\right)  ^{T}$, together with a routing matrix $P=\{P_{i,j}:i\leq
i,j\leq d\}$. Without loss of generality, we assume that $\sum_{i=1}%
^{d}\left(  {\lambda}_{i}+\mu_{i}\right)  =1$. The network is assumed to be
open and stable so conditions i), ii), and iii) described in the previous
section are in place.

Given the stability assumption, the system of equations given by
\begin{equation}
\phi_{i}={\lambda}_{i}+\sum_{j=1}^{d}\phi_{j}P_{ji},\quad\quad\forall
i=1,2,...,d
\end{equation}
admits a unique solution $\phi={\lambda}^{T}\left(  I-P\right)  ^{-1}$ (see
\cite{ASGLYNN}). The traffic intensity at station $i$ in the system in
equilibrium is given by $\rho_{i}$ which is defined by%
\begin{equation}
\rho_{i}=\frac{\phi_{i}}{\mu_{i}}=\frac{\left[  {\lambda}^{T}\left(
I-P\right)  ^{-1}\right]  _{i}}{\mu_{i}} \label{traffic intensity}%
\end{equation}
and satisfies $\rho_{i}\in\left(  0,1\right)  $ for all $i=1,2,...,d$. Define
$\rho_{\ast}=\max_{1\leq i\leq d}\rho_{i}$ and let $\beta$ be the cardinality
of the set $\{i:\rho_{i}=\rho_{\ast}\}$.

We shall study the queueing network by means of the embedded discrete time
Markov chain $Q=\{Q(k):k\geq0\}$, where $Q(k)=\left(  Q_{1}(k),\ldots
,Q_{d}(k)\right)  $. For each $k$, $Q_{i}(k)$ represents the number of
customers in station $i$ immediately after the $kth$ transition epoch of the
system. As mentioned before, the process $Q$ lives in the space $\mathcal{S}%
=\mathcal{Z}_{+}^{d}$.

Let $V\left(  y\right)  =y^{T}v$ be the total population in the stations
corresponding to the binary vector $v$. We are interested in the overflow
probability in any given subset of the Jackson network. More precisely, we
wish to estimate%
\begin{align*}
p_{n}^{V}  &  ={\mathbb{P}}\left\{  \text{ total population in stations
encoded by $v$ reaches}\right. \\
&  n\text{ before returning to $0$, starting from $0$}\}.
\end{align*}
In turn, $p_{n}^{V}$ can be expressed in terms of the following stopping
times,
\begin{align*}
&  T_{\left\{  0\right\}  }\triangleq\inf\{k\geq1:Q\left(  k\right)  =0\},\\
&  T_{n}^{V}\triangleq\inf\{k\geq1:V\left(  Q\left(  k\right)  \right)  \geq
n\}.
\end{align*}
Indeed, if we use the notation ${\mathbb{P}}_{x}(\cdot)\triangleq{\mathbb{P}%
}(\cdot|Q(0)=x)$ then we can rewrite $p_{n}^{V}$ as%
\begin{equation}
p_{n}^{V}={\mathbb{P}}_{0}(T_{n}^{V}\leq T_{\left\{  0\right\}  }).
\label{eqn:prob}%
\end{equation}
Similarly,
\begin{equation}
p_{n}^{V}\left(  x\right)  ={\mathbb{P}}_{x}(T_{n}^{V}\leq T_{\left\{
0\right\}  }). \label{p_x}%
\end{equation}

The asymptotic analysis of $p_{n}^{V}\left(  x\right)  $ can be studied by
means of large deviations theory. We shall indicate how this theory can be
applied to specify an efficient splitting algorithm in the next section. In
the mean time, let us provide a representation for the dynamics of the queue
length process that will be convenient in order to motivate the elements of
the efficient splitting algorithm that we shall analyze.

As we mentioned earlier, Jackson networks are basically constrained random
walks. The constraints arise because the number of customers in each station
must be non-negative. Thinking about Jackson networks as constrained random
walks facilitates the introduction and motivation of the necessary large
deviations elements behind the description of the splitting algorithm. In
order to specify the dynamics of the embedded discrete time Markov chain in
terms of a random walk type representation we need to introduce notation which
will be useful to specify the transitions at the boundaries induced by the
non-negativity constraints.

The state-space $\mathcal{Z}_{+}^{d}$ can be partitioned in $2^{d}$ different
regions which are indexed by all the subsets $E\subseteq\{1,\ldots,d\}$. The
region encoded by a given subset $E$ is defined as
\[
\partial_{E}=\{z\in{\mathbb{Z}}_{+}^{d}:z_{i}=0,i\in E,z_{i}>0,i\notin E\}.
\]
The interior of the domain is given by ${\partial_{\emptyset}}$ and the origin
is represented by $\partial_{\{1,2,...,d\}}$. Subsets other than the empty set
represent the \textquotedblleft boundaries\textquotedblright\ of the
state-space and correspond to system configurations in which at least one
station is empty. The collection of all possible values that the increments of
the process $Q$ can take depends on the current region at which $Q$ is
positioned. However, in any case, such collection is a subset of%
\[
{\mathbb{V}}\triangleq\{e_{i},-e_{i}+e_{j},-e_{j}:i,j=1,2,...,d\},
\]
where $e_{i}$ is the vector whose $i$-th component is one and the rest are
zero. An element of the form $e_{i}$ represents an arrival at station $i$, an
element of the form $-e_{i}+e_{j}$ represents a departure from station $i$
that flows to station $j$ and an element of the form $-e_{j}$ represents a
departure from station $i$ out of the system. The set of all possible
departures from station $i$ is a subset of%
\[
{\mathbb{V}}_{i}^{-}\triangleq\{w:w=-e_{i}\text{ or }w=-e_{i}+e_{j}\text{ for
some }j=1,...,d\}.
\]

\indent Because of the nonnegativity constraints on the boundaries of the
system we have to be careful when specifying the transition dynamics. First we
define a sequence of i.i.d. random variables $\{Y\left(  k\right)  :k\geq1\}$
so that for each $w\in{\mathbb{V}}$%
\[
\mathbb{P}\left(  Y\left(  k\right)  =w\right)  =\left\{
\begin{array}
[c]{ccccc}%
{\lambda}_{i} &  & \quad\text{if }w=e_{i}, &  & \\
\mu_{i}P_{ij} &  & \quad\text{if }w=-e_{i}+e_{j}, &  & \\
\mu_{i}P_{i0} &  & \quad\text{if }w=-e_{i}. &  &
\end{array}
\right.  .
\]
The dynamics of the queue-length process admit the random walk type
representation given by
\begin{equation}
Q(k+1)=Q(k)+\zeta\left(  Q(k),Y\left(  k+1\right)  \right)  ,
\label{eqn:dynamics}%
\end{equation}
where $\zeta\left(  \cdot\right)  $ is the constrained mapping and it is
defined for $x\in\partial_{E}$ via%
\[
\zeta\left(  x,w\right)  \triangleq\left\{
\begin{array}
[c]{ccccc}%
0 &  & \quad\text{if }w\in\cup_{i\in E}{\mathbb{V}}_{i}^{-} &  & \\
w &  & \quad\text{otherwise} &  &
\end{array}
\right.  .
\]

The large deviations theory associated to Jackson networks is somewhat similar
(at least in form) to that of random walks. One has to recognize, of course,
that the non-smoothness of the constrained mapping as a function of the state
of the system creates substantial technical complications, but we will leave
aside this issue in our discussion because our objective is simply to describe
the form of the necessary large deviations results for our purposes. An
extremely important role behind the development of large deviations theory for
light-tailed random walks is played by the logmoment generating function of
the increment distribution. So, given the similarities suggested by the
dynamics of (\ref{eqn:dynamics}) and those of a simple random walk it is not
surprising that the logmoment generating function of the increments, namely,%
\begin{equation}
\label{eqn:psi}\psi\left(  x,\theta\right)  \triangleq\log{\mathbb{E}}\left[
\exp\left(  \theta^{T}\zeta\left(  x,Y\left(  k\right)  \right)  \right)
\right]
\end{equation}
also plays a crucial role in the large deviations behavior of $p_{n}%
^{V}\left(  x\right)  $ as $n\nearrow\infty$.

In order to understand the large deviations behavior of $p_{n}^{V}$ it is
useful to scale space by $1/n$, thereby introducing a scaled queue length
process $\{Q_{n}\left(  k\right)  :k\geq0\}$ which evolves according to%
\[
Q_{n}(k+1)=Q_{n}(k)+\frac{1}{n}\zeta\left(  Q_{n}(k),Y\left(  k+1\right)
\right)  .
\]
Suppose that $Q_{n}\left(  0\right)  =y=x/n$ and note that
\[
T_{\left\{  0\right\}  }\triangleq\inf\{k\geq1:Q_{n}\left(  k\right)
=0\},T_{n}^{V}\triangleq\inf\{k\geq1:V\left(  Q_{n}\left(  k\right)  \right)
\geq1\}.
\]
Note that using the scaled queue length process one can write%
\begin{equation}
p_{n}^{V}\left(  y\right)  ={\mathbb{E}}\left[  p_{n}^{V}(y+\frac{1}{n}%
\zeta\left(  y,Y\left(  1\right)  \right)  )\right]  . \label{eqn:firststep}%
\end{equation}
Large deviations theory dictates that
\begin{equation}
p_{n}^{V}\left(  y\right)  =\exp\left(  -nW_{V}\left(  y\right)  +o\left(
n\right)  \right)  \label{eqn:largedevapp}%
\end{equation}
as $n\nearrow\infty$ for some non-negative function $W_{V}\left(
\cdot\right)  $. In order to characterize $W_{V}\left(  \cdot\right)  $ we can
combine the previous expression together with (\ref{eqn:firststep}) and a
formal Taylor expansion to obtain%
\begin{align*}
1  &  =\frac{1}{p_{n}^{V}\left(  y\right)  }{\mathbb{E}}\left[  p_{n}%
^{V}(y+\frac{1}{n}\zeta\left(  y,Y\left(  1\right)  \right)  )\right] \\
&  \approx\mathbb{E}\exp\{-nW_{V}[y+\frac{1}{n}\zeta\left(  y,Y\left(
1\right)  \right)  ]+nW_{V}\left(  y\right)  \}\\
&  =\mathbb{E}\exp\{-\partial W_{V}(y)\zeta\left(  y,Y\left(  1\right)
\right)  +o\left(  1\right)  \}\\
&  =\exp\left(  \psi\left(  y,-\partial W_{V}\left(  y\right)  \right)
+o\left(  1\right)  \right)  .
\end{align*}
Sending $n\nearrow\infty$ we formally arrive at the equation
\begin{equation}
\psi\left(  y,-\partial W_{V}\left(  y\right)  \right)  =0 \label{eq:Isaacs}%
\end{equation}
together with the boundary condition $W_{V}\left(  y\right)  =0$ if $V\left(
y\right)  \geq1$. The previous equation is the so-called Isaacs equation which
characterizes the large deviations behavior of $p_{n}^{V}\left(  \cdot\right)
$ and it was introduced together with a game theoretic interpretation by
Dupuis and Wang in \cite{DUPWAN04}. The solution to (\ref{eq:Isaacs}) is
understood in a weak sense because the function $W_{V}\left(  \cdot\right)  $
is typically not differentiable everywhere. Nevertheles, it coincides with a
certain calculus of variations representation which can be obtained out of the
local large deviations rate function for Jackson networks (see
\cite{MajRamanan08}).

An asymptotic lower bound for $W_{V}\left(  y\right)  $ can be obtained by
finding an appropriate subsolution to the Isaacs equation, in which the
equality signs in (\ref{eq:Isaacs}) are appropriately replaced by inequalities
thereby obtaining a so-called subsolution to the Isaacs equation. In
particular, $\overline{W}_{V}\left(  \cdot\right)  $ is said to be a
subsolution to the Isaacs equation if
\begin{equation}
\psi(y,-\partial\overline{W}_{V}\left(  y\right)  )\leq0 \label{eqn:subIsaacs}%
\end{equation}
subject to $\overline{W}_{V}\left(  y\right)  \leq0$ if $V\left(  y\right)
\geq1$. The subsolution property guarantees $\overline{W}_{V}\left(  y\right)
\leq W_{V}\left(  y\right)  $, which translates to an asymptotic logarithmic
upper bound $p_{n}^{V}\left(  y\right)  $. The subsolution is said to be
maximal at zero if $\overline{W}_{V}\left(  0\right)  =W_{V}\left(  0\right)
$. Not surprisingly, subsolutions are easier to construct than solutions and,
as we shall discuss in the next section, beyond their use in the development
of asymptotic upper bounds they can be applied to the design of efficient
simulation procedures. The use of subsolutions to the Isaacs equation for the
design of efficient simulation algorithms was introduced in \cite{DUPWAN04}. A
derivation of the subsolution equation (\ref{eqn:subIsaacs}) following the
same spirit leading to (\ref{eq:Isaacs}) using Lyapunov inequalities is given
in \cite{BlanLedGlyn09}.

As we mentioned in Section 2, the efficiency analysis of a rare-event
simulation estimator depends on the growth rate of its coefficient of
variation. We are interested in an asymptotic analysis that goes beyond the
error term $\exp(o\left(  n\right)  )$ given by the large deviations
approximation (\ref{eqn:largedevapp}). So, we must enhance the large
deviations approximations in order to provide a more precise estimate for
$p_{n}^{V}$. Developing such an estimate is the aim of the following
proposition which follows as a direct consequence of Proposition 2 and the
analysis in Section 5 in \cite{Blanchet09}; see also Section 4 in this paper
for a sketch of the proof.

\begin{proposition}
There exists $K>0$ (independent of $x$ and $n$) such that%
\[
p_{n}^{V}\left(  x\right)  \leq KP\{V\left(  Q\left(  \infty\right)  \right)
=n\}/P\{Q\left(  \infty\right)  =x\},
\]
where $Q_{\infty}$ is the steady state queue length. Moreover, if $\left\Vert
x\right\Vert \leq c$ for some $c\in(0,\infty)$ then$\footnote{Given two
non-negative functions $f\left(  \cdot\right)  $ and $g\left(  \cdot\right)
$, we say $f\left(  n\right)  =\Omega\left(  g\left(  n\right)  \right)  $ if
there exists $c,n_{0}\in(0,\infty)$ such that $f\left(  n\right)  \geq
cg\left(  n\right)  $ all $n\geq n_{0}$.}$%
\[
p_{n}^{V}\left(  x\right)  =\Omega\lbrack P\{V\left(  Q\left(  \infty\right)
\right)  =n\}/P\{Q\left(  \infty\right)  =x\}]
\]
as $n\nearrow\infty$.
\end{proposition}

\textbf{Remark }It is important to keep in mind that we shall mostly work with
the process $Q\left(  \cdot\right)  $ directly, as opposed to the scaled
version $Q_{n}\left(  \cdot\right)  $ which is used in the analysis of
\cite{DeanDup09}.

\bigskip

The previous proposition provides the necessary means to estimate $p_{n}^{V}$
up to a constant; we just need to recall that the distribution of $Q\left(
\infty\right)  $ is computable in closed form (see \cite{RobertQN}\ p. 95). In
particular, we have that
\begin{align*}
\pi\left(  m_{1},...,m_{d}\right)   &  =\prod_{j=1}^{d}{\mathbb{P}}\left(
Q_{j}\left(  \infty\right)  =m_{j}\right) \\
&  =\prod_{j=1}^{d}\left(  1-\rho_{j}\right)  \rho_{j}^{m_{j}},\quad
j=1,...,d,\text{ and }m_{j}\geq0.
\end{align*}
We shall use $\pi\left(  \cdot\right)  $ to denote the stationary measure of
$Q$. In simple words, the previous equation says that the steady state queue
length process has independent components which are geometrically distributed.
In particular, $P\left(  Q_{j}\left(  \infty\right)  =m\right)  =\rho_{j}%
^{m}(1-\rho_{j})$ for $m\geq0$. The next proposition follows directly from
standard properties of the geometric distribution (see Proposition 3 in
\cite{Blanchet09}).

\begin{proposition}
\label{prop: prob aym} $P[V\left(  Q\left(  \infty\right)  \right)
=n]=\Theta\left(  e^{-n\gamma_{V} }n^{\beta_{V}-1}\right)  \footnotetext[3]%
{Given two positive functions $f\left(  \cdot\right)  $ and $g\left(
\cdot\right)  $, recall that $f\left(  n\right)  =\Theta\left(  g\left(
n\right)  \right)  $ if $f\left(  n\right)  =O\left(  g\left(  n\right)
\right)  $ and $f\left(  n\right)  =\Omega\left(  g\left(  n\right)  \right)
$.}$, where $\gamma_V=-\log\rho^V_\ast$, in which $\rho^V_\ast=\max\{\rho_i:
v_i=1\}$; and $\beta_V=\sum_iI_\{\rho_i=\rho^V _\ast,v_i=1\}$ is the number of
bottleneck stations in the target subset corresponding to $v$.
\end{proposition}

\section{The Splitting Algorithm}

\label{splitting}

The previous section discussed some large deviations properties required to
guide the construction of an efficient splitting scheme using the theory
developed in the work of Dean and Dupuis \cite{DeanDup09}. In order to explain
the construction suggested by Dean and Dupuis let us first discuss the general
idea behind the splitting algorithm that we shall analyze; a variation of
which was first applied to Jackson networks by by Villen-Altamirano and
Villen-Altamirano \cite{VAVA91}.

The strategy is to divide the state-space into a collection of regions
$\{C_{j}^{n}:0\leq j\leq l_{n}\left(  x\right)  \}$ which are nested and that
help define \textquotedblleft milestone\textquotedblright\ events that
interpolate between the initial position of the process and the target set,
which corresponds to the region $C_{0}^{n}$. That is, in our setting we put
$C_{0}^{n}\triangleq\{y\in\mathcal{S}:V\left(  y\right)  \geq n\}$ and the
remaining $C_{j}^{n}$'s are placed so that $C_{0}^{n}\subseteq C_{1}%
^{n}\subseteq...\subseteq C_{M_{n}}^{n}$. How to construct the level sets
$C_{j}^{n}$ in order to induce efficiency will be discussed below. An
observation that is intuitive at this point, however, is that one should have
$M_{n}=\Theta\left(  n\right)  $ so that the next milestone event becomes
accessible given the current level. For the moment, let us assume that the
$C_{j}^{n}$'s have been placed. The splitting algorithm proceeds as follows.

\bigskip

\textbf{Algorithm SA}

1.- Initiate the simulation procedure with a single particle starting from
position $x\in C_{k}^{n}$ for a given $k\geq1$. Let $w_{1}=1$ be the initial
weight associated to such particle.

2.- Evolve the initial particle until either it hits $\{0\}$ or it hits level
$C_{k-1}^{n}$. If the particle hits $\{0\}$, then the particle is said to die.
If the particle reaches level $C_{k-1}^{n}$ then it is \textit{replaced} by
$r$ identical particles (for a given integer $r>1$). The replacing particles
are called the immediate descendants or children of the initial particle,
which in turn is said to be their parent. The children are positioned
precisely at the place where the parent particle reached level $C_{k-1}^{n}$.
The weight $w_{j}$ associated to the $j$-th children (enumerate the children
arbitrarily) has a value equal to the weight of the parent particle multipled
by $1/r$.

3.- The procedure starting from step 1 is replicated for each of the offspring
particles in place; carrying over the value of each of the weights at each
level for the surviving particles (the weights of the particles that die can
be disregarded).

4.- Steps 1 to 3 are repeated until all the particles either die or reach
level $C_{0}^{n}$.

\bigskip

Dean and Dupuis in \cite{DeanDup09} show how to apply large deviations theory
to select the $C_{j}^{n}$'s in order to obtain a weakly efficient splitting
algorithm. One needs to balance the number of the $C_{j}^{n}$'s so that it is
not unlikely for a given particle to reach the next level while keeping the
total number of particles controlled. We now provide a formal motivation for
the use of large deviations for constructing the $C_{j}^{n}$'s in a balanced way.

It is convenient, as we did in our formal large deviations discussion in the
previous section, to consider the scaled process $Q_{n}\left(  \cdot\right)
$. Let us assume that the splitting mechanism indicated in Algorithm SA is in
place and that our initial position is set at level $Q\left(  0\right)  =x$,
so that $Q_{n}\left(  0\right)  =y=x/n$. The $C_{j}^{n}$'s are typically
constructed in terms of the level sets of a so-called importance function
which we shall denote by $U\left(  \cdot\right)  $. In particular, put
$D_{n}\triangleq\{y\in n^{-1}\mathcal{S}:V\left(  y\right)  <1\}$ and set
$C_{j}^{n}=nL_{z_{n}\left(  j\right)  }$, where
\begin{equation}
L_{z}\triangleq\{y\in D_{n}:U\left(  y\right)  \leq z\}, \label{eqn:level fun}%
\end{equation}
and the $z_{n}\left(  j\right)  $'s are appropriately chosen momentarily.
Then, define
\begin{equation}
l_{n}\left(  y\right)  =\min\{j\geq0:ny\in C_{j}^{n}\}=\min\{j\geq0:x\in
C_{j}^{n}\}. \label{eq:def_ly}%
\end{equation}
The total weight corresponding to a particle that reaches level $C_{0}^{n}$
given that it started at level $l_{n}\left(  y\right)  $ is $r^{-l_{n}\left(
y\right)  }$. In order to have at least a weakly efficient algorithm we wish
to achieve two constraints. The first one imposes the total weight of a
particle reaching level $C_{0}^{n}$ to be $p_{n}^{V}\left(  x\right)
\exp\left(  o\left(  n\right)  \right)  $; this would guarantee that the
second moment of the resulting estimator achieves asymptotic optimality. The
second constraint dictates that the expected number of particles that make it
to $C_{0}^{n}$, which is roughly $r^{l_{n}\left(  y\right)  }p_{n}^{V}\left(
x\right)  $ exhibits subexponential growth (i.e. $\exp\left(  o\left(
n\right)  \right)  $); this would guarantee a cost per replication that is
subexponential. Note that both constraints lead to the requirement of
$r^{l_{n}\left(  y\right)  }p_{n}^{V}\left(  x\right)  =\exp\left(  o\left(
n\right)  \right)  $. So, given a subsolution $\overline{W}_{V}\left(
\cdot\right)  $ to the corresponding Isaacs equation, which implies that
\[
p_{n}^{V}\left(  x\right)  \leq\exp\left(  -n\bar{W}_{V}\left(  y\right)
+{o}\left(  n\right)  \right)  ,
\]
it suffices to ensure that%
\begin{equation}
l_{n}\left(  y\right)  \log\left(  r\right)  -n\overline{W}_{V}\left(
y\right)  =o\left(  n\right)  . \label{eq:heur_sub}%
\end{equation}
The behavior of $l_{n}\left(  y\right)  $ as a $n\nearrow\infty$ only relates
to the properties of the function $U\left(  \cdot\right)  $ and it is really
independent of the large deviations behavior of the system. In particular,
picking $z_{n}\left(  j\right)  =\Delta j/n,\Delta\in\left(  0,1\right]  $
yields $l_{n}\left(  y\right)  =\left\lceil nU\left(  y\right)  /\Delta
\right\rceil $ and therefore, equation (\ref{eq:heur_sub}) suggests that one
should select $U\left(  y\right)  =\Delta\overline{W}_{V}\left(  y\right)
/\log\left(  r\right)  $ with $\overline{W}_{V}\left(  0\right)  =W_{V}\left(
0\right)  $ in order to obtain a weakly efficient estimator for $p_{n}^{V}$.
This is precisely the conclusion obtained in the work of \cite{DeanDup09} who
present a rigorous analysis that justifies the previous heuristic discussion.
Our development in the next section will sharpen the efficiency properties of
the sampler proposed in \cite{DeanDup09} when applied to Jackson networks. So,
we content ourselves with the previous heuristic motivation for the splitting
method that we will analyze in the next section and which in turn is based on
the viscosity subsolution given by
\begin{equation}
\bar{W}_{V}\left(  y\right)  =\varrho^{T}y-\log\rho_{\ast}^{V},
\label{eqn:usefulsub}%
\end{equation}
where $\varrho_{i}=-\log\rho_{i}$ for $i=1,...,d$, see e.g., \cite{DupWang08}
and \cite{DeanDup09}.

We close this section with a precise definition of the estimator that we will
analyze. First, given a constant $\Delta\in\left(  0,1\right]  $ (the level
size) define $\bar{W}_{V}\left(  \cdot\right)  $ as indicated in
(\ref{eqn:usefulsub}) for each $y=x/n$ with $x\in\mathcal{S}$. Then, select an
integer $r>1$ and define $U\left(  y\right)  =\Delta\overline{W}_{V}\left(
y\right)  /\log\left(  r\right)  $. Given the initial position $x$ put $y=x/n$
and define the sets $\{C_{j}^{n}:1\leq j\leq l_{n}\left(  y\right)  \}$ as
indicated above (see equation (\ref{eq:def_ly})). Run Algorithm SA and let
$N_{n}$ be the number of particles that survive up to $C_{0}^{n}$; their
corresponding final weight is $1/r^{l_{n}\left(  x\right)  }$. Our estimator
for $p_{n}^{V}\left(  x\right)  $ is simply
\begin{equation}
R_{n}\left(  x\right)  =N_{n}\left(  x\right)  /r^{l_{n}\left(  x\right)  }
\label{eqn:SA estimator}%
\end{equation}

Now, for the sake of analytical convenience, when analyzing the second moment
of $R_{n}\left(  x\right)  $ we will adopt the so-called \emph{fully
branching} representation of the previous estimator (see \cite{DeanDup09}).
Such fully branching representation is obtained by splitting death particles
at level zero. It is useful to think about fully branching conceptually in the
following terms. Start with a particle at position $x$ in level $C_{l_{n}%
\left(  x\right)  }^{n}$ and run step 2 of Algorithm SA, but instead of
killing the particle if it hits $\{0\}$ before reaching $C_{l_{n}\left(
x\right)  -1}^{n}$, just allow it to also split into $r$ particles and update
the weight of the children as indicated in Algorithm SA when the particle
reaches $C_{l_{n}\left(  x\right)  -1}^{n}$. Step 3 continues, now the
particles that are sitting in level $C_{l_{n}\left(  x\right)  -1} ^{n}$ will
evolve and the death particles will once again split and remain in state $0$
(so state $0$ is being populated with death particles). After $l_{n}\left(
x\right)  $ iterations we have $r^{l_{n}\left(  x\right)  }$ total particles
labeled $1,2,...,r^{l_{n}\left(  x\right)  }$, each with weight $1/r^{l_{n}%
\left(  x\right)  }$. We define $I_{j}$ as the indicator function of the event
that the $j$-th particle is in $C_{0}^{n}$ so that $N_{n}\left(  x\right)
=\sum_{j=1}^{r^{l_{n}\left(  x\right)  }}I_{j}$. The fully branching
representation of $R_{n}\left(  x\right)  $ is simply%
\begin{equation}
R_{n}\left(  x\right)  ={r^{-l_{n}\left(  x\right)  }}\sum_{j=1}%
^{r^{l_{n}\left(  x\right)  }}I_{j}. \label{eqn:SFB estimator}%
\end{equation}

\section{Analysis of Splitting Estimators}

\label{analysis} We are now in a good position to perform a refined efficiency
analysis for the estimator $R_{n}\left(  x\right)  $. We shall break our
analysis into two parts. The first part corresponds to the expected number of
particles generated per run and the second part deals with the second moment
of $R_{n}\left(  x\right)  $. First, we will review a technique studied in
\cite{Blanchet09} based on the corresponding time reversed chain associated to
the queue length process Q. For both quantities we are able to obtain an upper
bound. We are then able to reach the conclusion that this multilevel Splitting
algorithm substantially outperforms the direct polynomial time algorithm for
solving the associated system of linear equations.

Our analysis takes advantage of the time reversed process associated to the
underlying Jackson network which we shall now define. Given the transition
matrix $\{K\left(  x,y\right)  :x,y\in\mathcal{S}\}$ of the process $Q$, we
define the reversed Markov chain $\tilde{Q}=\{\tilde{Q}\left(  k\right)
:k\geq0\}$ via the transition matrix $\tilde{K}\left(  \cdot\right)  $:
\[
\tilde{K}\left(  y,x\right)  =K\left(  x,y\right)  \pi\left(  x\right)
/\pi\left(  y\right)  ,
\]
for $x,y\in\mathcal{S}$. It turns out that $\tilde{Q}$ also describes the
queue length process of an open stable Jackson network with stationary
distribution equal to $\pi\left(  \cdot\right)  $, (see \cite{RobertQN}\ p.
95). We will use $\tilde{P}_{x}\left(  \cdot\right)  $ to denote the
probability measure in path space associated to $\tilde{Q}$ given that
$\tilde{Q}\left(  0\right)  =x$.\newline\indent The following result is
similar to that of Proposition 2 in \cite{Blanchet09}. However, our
representation in (\ref{eqn:rep_overflow}) is slightly more useful for our purposes.

\begin{proposition}%
\begin{align}
p_{n}^{V}\left(  x\right)   &  =\frac{\tilde{{\mathbb{P}}}_{\pi}\left(
\tilde{Q}\left(  0\right)  \in C_{0}^{n},\tilde{T}_{\{x\}}\leq\tilde
{T}_{\left\{  0\right\}  },\tilde{T}_{\{x\}}<\tilde{T}_{C_{0}^{n}}\right)
}{\pi(x)P_{x}(T_{\{x\}}\geq T_{C_{0}^{n}}\wedge T_{\{0\}})}%
\label{eqn:rep_overflow}\\
&  =\frac{\tilde{{\mathbb{P}}}_{\pi}\left(  \tilde{Q}\left(  0\right)  \in
C_{0}^{n},\tilde{T}_{\{x\}}\leq\tilde{T}_{\left\{  0\right\}  }<\tilde
{T}_{C_{0}^{n}}\right)  }{\pi(x)P_{0}(T_{\{x\}}\leq T_{C_{0}^{n}}\wedge
T_{\{0\}})} \label{eqn:rep_overflow2}%
\end{align}
where $\tilde{T}_{C_{0}^{n}}\triangleq\inf\{k\geq1:\tilde{Q}\left(  k\right)
\in C_{0}^{n}\}$, $\tilde{T}_{\{x\}}\triangleq\inf\{k\geq0:\tilde{Q}\left(
k\right)  =x\}$, $T_{C_{0}^{n}}\triangleq\inf\{k\geq1:Q\left(  k\right)  \in
C_{0}^{n}\}$ and $T_{\{x\}}=\inf\{k\geq0:Q\left(  k\right)  =x\}$. Moreover,
there exists $\delta>0$ (independent of $x$ and $n$) such that%
\begin{equation}
P_{x}(T_{\{x\}}\geq T_{C_{0}^{n}}\wedge T_{\{0\}})\geq P_{x}(T_{\left\vert
\left\vert x\right\vert \right\vert }\geq T_{\{0\}})\geq\delta,
\label{Eq:Bndrep}%
\end{equation}
where $\left\vert \left\vert x\right\vert \right\vert $ is the $L_{1}$ norm of
$x$.
\end{proposition}

\begin{proof}
We assume that $x\neq0$. The case $x=0$ is included in the analysis
of (\ref{eqn:rep_overflow2}).
First, we observe that%
\begin{align*}
p_{n}^{V}\left(  x\right)    & =P_{x}\left(  T_{C_{0}^{n}}<T_{\{0\}}%
,T_{\{x\}}<T_{C_{0}^{n}}\wedge T_{\{0\}}\right)  +P_{x}\left(  T_{C_{0}^{n}%
}<T_{\{0\}},T_{\{x\}}\geq T_{C_{0}^{n}}\wedge T_{\{0\}}\right)  \\
& =p_{n}^{V}\left(  x\right)  P_{x}\left(  T_{\{x\}}<T_{C_{0}^{n}}\wedge
T_{\{0\}}\right)  +P_{x}\left(  T_{C_{0}^{n}}<T_{\{0\}},T_{\{x\}}\geq
T_{C_{0}^{n}}\wedge T_{\{0\}}\right)  .
\end{align*}
Therefore,
\[
p_{n}^{V}\left(  x\right)  =\frac{P_{x}\left(  T_{C_{0}^{n}}<T_{\{0\}}%
,T_{\{x\}}\geq T_{C_{0}^{n}}\wedge T_{\{0\}}\right)  }{P_{x}\left(
T_{\{x\}}\geq T_{C_{0}^{n}}\wedge T_{\{0\}}\right)  }.
\]
Following the same technique as in Proposition 2 in \cite{Blanchet09} we have
that%
\begin{align*}
& \pi\left(  x\right)  P_{x}\left(  T_{C_{0}^{n}}<T_{\{0\}},T_{\{x\}}\geq
T_{C_{0}^{n}}\wedge T_{\{0\}}\right)  \\
& =\sum_{k=0}^{\infty}\pi\left(  x\right)  P_{x}\left(  T_{C_{0}^{n}%
}<T_{\{0\}},T_{\{x\}}\geq T_{C_{0}^{n}}\wedge T_{\{0\}},T_{C_{0}^{n}%
}=k\right)  \\
& =\sum_{k=1}^{\infty}\pi\left(  x\right)  \sum_{y_{0}=x,y_{1},..,y_{k}%
}K\left(  y_{0},y_{1}\right)  \times...\times K\left(  y_{k-1},y_{k}\right)
I\left(  k<T_{\{0\}},T_{\{x\}}\geq
k\wedge T_{\{0\}},T_{C_{0}^{n}}=k\right)  \\
& =\sum_{k=1}^{\infty}\sum_{y_{0}=x,y_{1},..,y_{k}}K^{\prime}\left(
y_{1},y_{0}\right)  \times...\times K^{\prime}\left(  y_{k},y_{k-1}\right)
\pi\left(  y_{k}\right)  I\left(  k<T_{\{0\}},T_{\{x\}}\geq k\wedge T_{\{0\}},T_{C_{0}^{n}}=k\right)  .
\end{align*}
Letting $y_{i}^{\prime}=y_{k-i}$ for $i=1,...k$ we see that the summation in
each of the terms above ranges over paths $y_{0}^{\prime},...,y_{k}^{\prime}$
satisfying that $y_{0}^{\prime}\in C_{0}^{n}$, $\tilde{T}_{\{x\}}=k$ (so in
particular $y_{k}^{\prime}=x$) and also that $\tilde{T}_{\left\{  0\right\}
}\geq k, \tilde{T}_{C_{0}^{n}} > k$. So, we can interpret the previous sum as%
\[
\tilde{{\mathbb{P}}}_{\pi}\left(  \tilde{Q}\left(  0\right)  \in C_{0}%
^{n},\tilde{T}_{\{x\}}\leq\tilde{T}_{\left\{  0\right\}  }, \tilde{T}_{\{x\}} < \tilde
{T}_{C_{0}^{n}}\right)  .
\]
This yields part (\ref{eqn:rep_overflow}). Part (\ref{eqn:rep_overflow2})
corresponds to Proposition 2 of \cite{Blanchet09}. Finally, (\ref{Eq:Bndrep})
also follows as in Proposition 7 of \cite{Blanchet09}.
\end{proof}

\noindent Proposition 1 and 2 from Section 2 follow as a consequence of this
result, the rest of the details are given in Section 5 of \cite{Blanchet09}%
.\smallskip

Given the subsolution we proposed in Section $4$, the importance function can
be written as
\begin{align}
U\left(  x/n\right)  =\bar{W}_{V}\left(  x/n\right)  \frac{\Delta}{\log r}  &
=\left(  \frac{1}{n}\varrho^{T}x-\log\rho_{\ast}^{V}\right)  \frac{\Delta
}{\log r}\nonumber\label{eqn:importance function}\\
&  =C\left(  \Delta-\frac{1}{n}\alpha^{T}x\Delta\right)  ,
\end{align}
where $C=-\log\rho_{\ast}^{V}/\log r$, and $\alpha=\varrho\,/\log\rho_{\ast
}^{V}$. The level index function also simplifies to
\begin{equation}
l^{n}\left(  x\right)  =\lceil{\frac{nU\left(  x/n\right)  }{\Delta}}%
\rceil=\lceil nC\left(  1-\frac{1}{n}\alpha^{T}x\right)  \rceil=\lceil
{C\left(  n-\alpha^{T}x\right)  }\rceil. \label{eqn:level ind fun}%
\end{equation}
We shall first look at the expected number of surviving particles of the
splitting algorithm which characterizes the stability of the algorithm. One
shall keep in mind that when the complexity of the splitting algorithm is
concerned, what actually matters is the total function evaluation involved in
each run. An upper bound is obtained for this quantity, as measured by the sum
of all particles generated at interim levels weighted by the maximum remaining
function evaluations associated with each of them. We first have the following result.

\begin{proposition}
The expected terminal number of particles for the splitting algorithm
specified by $\left(  \Delta, U\right)  $ above satisfies
\begin{equation}
\label{eqn:num particles}{\mathbb{E}}\left[  N_{n}\left(  x\right)  \right]
=\Theta\left(  \thinspace n^{\beta_{V}-1}\right)
\end{equation}
where $\beta_{V}$, introduced in Proposition 2, denotes the number of
bottleneck stations corresponding to the vector $v$.
\end{proposition}

\begin{proof}
Note that
\[
{\mathbb{E}}\left[ N_{n}\left( x\right) \right] =r^{l^{n}\left( x\right)
}\thinspace p_{n}^{V}\left( x\right) .
\]
In Proposition \ref{prop: prob aym} we know that $p_{n}^{V}\left( x\right)
=\Theta\left( e^{-\gamma_{V} n}n^{\beta_{V}-1}\right) $. Since $e^{-\gamma
_{V}}=e^{\log\rho^{V}_{*}}=e^{-C\log r}=r^{-C}$, we can write $p_{n}^{V}\left(
x\right) =\Theta\left( r^{-nC}n^{\beta_{V}-1}\right) $. Hence, plug in
$l^{n}\left( x\right) =\lceil C\left( n-x^{T}v\right) \rceil$, we have
\[
{\mathbb{E}}\left[ N_{n}\left( x\right) \right] =\Theta\left( r^{-nC}%
n^{\beta_{V}-1}r^{\lceil nC\rceil}\right) =\Theta\left(  \thinspace
n^{\beta_{V}-1}\right) .
\]
\end{proof}

As pointed out earlier, the number of terminal surviving particles, although
serves as a reasonable proxy to measure the stability of the algorithm, is not
suitable for quantifying the complexity. We also need to take into account the
number of function evaluations required to generate $R_{n}\left(  x\right)  $.
The next result addresses precisely this issue.

\begin{proposition}
\label{prop: effort per run} The expected computational effort per run
required to generate a single replication of $R_{n}\left(  x\right)  $ is
$O\left(  n^{\beta_{V}+1}\right)  $.
\end{proposition}

\begin{proof}
To see this, let $N_{m}^{n},\,m=0,....,l^{n}\left(  x\right)  $, be
the number of particles that survive to level $C_{l^{n}\left(
x\right)  -m}^{n}$. Also let $\eta_{m,j}$ be the remaining
computational effort of the $j$-th particle at the start of the
$m$-th level until it either reaches the next level or it dies out.
Put $\bar{\eta}_{m,j}\left(x_j\right)$ to be the expectation of
$\eta_{m,j}$ given that the position of the $j$-th particle at the
start of level $m$ is $x_j$. Note that the norm of the position of
$x_j$ is less than $c \cdot m$ for a given constant $c$ that depends
on the traffic intensities of the system but not on the position of
the particle per-se. Therefore, it is easy to see that
\[
\sup_{1\leq j\leq N^n_{m}}\bar{\eta}_{m,j}\left(x_j\right)\leq
c\cdot m,
\]
for some $c\in\left(0,\infty\right)$. Intuitively, each particle at level $m$ either advances to the next
level, or it dies out by hitting the zero level before moving to the
next one, since it takes $\Theta\left( 1\right) $ work to cross one
single layer, $\eta_{m,j}$ is dominated by the work required to die
out, and hence its mean is bounded from above by $c \times m$ for some
constant $c$. The expected total work per run is then given by
\begin{align*}
{\mathbb{E}}\left[  \sum_{m=0}^{l_{n}\left(  x\right)  -1}\sum_{j=1}%
^{N_{m}^{n}}\eta_{m,j}\right] &
=\sum^{l_n\left(x\right)-1}_{m=0}\E\left[\sum^{N^n_m}_{j=1}\bar{\eta}_{m,j}\left(x_j\right)\right]
\\
& \leq \sum_{m=0}^{l_{n}\left(  x\right)  -1}\E\left[N^n_m\right]\cdot c\cdot m \\
& \leq c\cdot  \sum_{m=0}^{l_{n}\left(  x\right)  -1}\left(
\frac{m}{C}\right)
^{\beta_{V}-1}\left(  \rho_{\ast}^{V}\right)  ^{\frac{m}{C}}r^{m}m  \\
& \leq c\cdot  \sum_{m=0}^{l_{n}\left(  x\right)  -1}\left(
\frac{m}{C}\right) ^{\beta_{V}}  =O\left( n^{\beta_{V}+1}\right) ,
\end{align*}
for some positive constant $c$ where the first inequality is due to
independence.
\end{proof}

To facilitate the analysis of the second moment of $R_{n}\left(  x\right)  $
we add the following notations. We follow the analysis in \cite{DeanDup09} to
make our exposition here self-contained. For a given generation $m$, denote by
$Q_{m,j}$ the position of the $j$th particle; recall that the accumulated
weight up to the $m$th stage of such a particle is $r^{m}$. Let $\chi_{m,j}$
be the disjoint grouping of particles in the next generation (i.e., $m+1$)
according to their \textquotedblleft parents\textquotedblright\ in generation
$m$. For $k\in\chi_{m,j}$, denote by $d_{k}$ the offsprings of this particle
at the final stage $l_{n}\left(  x\right)  $. We then have the following
expansion of the second moment of $R_{n}\left(  x\right)  $:
\begin{align}
&  {\mathbb{E}}_{x}\left[  \left(  \sum_{j=1}^{r^{l_{n}\left(  x\right)  }%
}I_{j}\,{r^{-l_{n}\left(  x\right)  }}\right)  ^{2}\right]
\label{eqn:expansion of 2nd mom}\\
&  =\sum_{m=0}^{l_{n}\left(  x\right)  -1}{\mathbb{E}}_{x}\left[  \sum
_{j=1}^{r^{m}}\,\sum_{k,l\in\chi_{m,j},k\neq l}\left(  \sum_{m_{k}\in d_{k}%
}I_{m_{k}}r^{-l_{n}\left(  x\right)  }\right)  \left(  \sum_{m_{l}\in d_{l}%
}I_{m_{l}}r^{-l_{n}\left(  x\right)  }\right)  \right]  +{\mathbb{E}}%
_{x}\left[  \sum_{j=1}^{r^{l_{n}\left(  x\right)  }}I_{j}\,{r^{-2l_{n}\left(
x\right)  }}\right]  ,\nonumber
\end{align}
where we define $I_{m_{k}}$ to be the indicator function of the event that
particle $m_{k}$ is in the set $C_{0}^{n}$. The second term above is
essentially the diagonal terms of the second moment
(\ref{eqn:expansion of 2nd mom}), and for the off-diagonal terms, for each
generation, we categorize particles according to their common ancestors, a
technique used by \cite{DeanDup09}. For the first term, we have
\begin{align}
&  \sum_{m=0}^{l_{n}\left(  x\right)  -1}{\mathbb{E}}_{x}\left[  \sum
_{j=1}^{r^{m}}\,\sum_{k,l\in\chi_{m,j},k\neq l}\left(  \sum_{m_{k}\in d_{k}%
}I_{m_{k}}r^{-l_{n}\left(  x\right)  }\right)  \left(  \sum_{m_{l}\in d_{l}%
}I_{m_{l}}r^{-l_{n}\left(  x\right)  }\right)  \right]
\nonumber\label{eqn:expand first term1}\\
&  =\sum_{m=0}^{l_{n}\left(  x\right)  }{\mathbb{E}}_{x}\left[  \sum
_{j=1}^{r^{m}}I\left(  V\left(  Q_{m,j}\right)  >0\right)  \left(
r^{-m}\right)  ^{2}\right. \nonumber\\
&  \cdot\left.  \sum_{k,l\in\chi_{m,j},k\neq l}\left(  \frac{1}{r}\sum
_{m_{k}\in d_{k}}I_{m_{k}}r^{-\left(  l_{n}\left(  x\right)  -m-1\right)
}\right)  \left(  \frac{1}{r}\sum_{m_{l}\in d_{l}}I_{m_{l}}r^{-\left(
l_{n}\left(  x\right)  -m-1\right)  }\right)  \right]  .\nonumber
\end{align}
Conditioning on the whole genealogy up to step $m$, we obtain
\begin{align*}
&  {\mathbb{E}}_{x}\left[  \sum_{j=1}^{r^{m}}I\left(  V\left(  Q_{m,j}\right)
>0\right)  \left(  r^{-m}\right)  ^{2}\right. \\
&  \cdot\left.  \sum_{k,l\in\chi_{m,j},k\neq l}\left(  \frac{1}{r}\sum
_{m_{k}\in d_{k}}I_{m_{k}}r^{-\left(  l_{n}\left(  x\right)  -m-1\right)
}\right)  \left(  \frac{1}{r}\sum_{m_{l}\in d_{l}}I_{m_{l}}r^{-\left(
l_{n}\left(  x\right)  -m-1\right)  }\right)  \right] \\
&  ={\mathbb{E}}_{x}\left[  \sum_{j=1}^{r^{m}}I\left(  V\left(  Q_{m,j}%
\right)  >0\right)  \left(  r^{-m}\right)  ^{2}\right. \\
&  \cdot\left.  {\mathbb{E}}\left(  \sum_{k,l\in\chi_{m,j},k\neq l}\left(
\frac{1}{r}\sum_{m_{k}\in d_{k}}I_{m_{k}}r^{-\left(  l_{n}\left(  x\right)
-m-1\right)  }\right)  \left(  \frac{1}{r}\sum_{m_{l}\in d_{l}}I_{m_{l}%
}r^{-\left(  l_{n}\left(  x\right)  -m-1\right)  }\right)  \right)  \right] \\
&  ={\mathbb{E}}_{x}\left[  \sum_{j=1}^{r^{m}}I\left(  V\left(  Q_{m,j}%
\right)  >0\right)  r^{-2m}\right. \\
&  \cdot\left.  \sum_{k,l\in\chi_{m,j},k\neq l}\left(  \frac{1}{r}{\mathbb{E}%
}_{\bar{X}_{m,j}}\left(  \sum_{m_{k}\in d_{k}}I_{m_{k}}r^{-\left(
l_{n}\left(  x\right)  -m-1\right)  }\right)  \frac{1}{r}{\mathbb{E}}_{\bar
{X}_{m,j}}\left(  \sum_{m_{l}\in d_{l}}I_{m_{l}}r^{-\left(  l_{n}\left(
x\right)  -m-1\right)  }\right)  \right)  \right]  .
\end{align*}
Note that ${\mathbb{E}}_{Q_{m,j}}\left[  \sum_{m_{k}\in d_{k}}I_{m_{k}%
}r^{-\left(  l_{n}\left(  x\right)  -m-1\right)  }\right]  =p_{n}^{V}\left(
Q_{m,j}\right)  $, and $\mathcal{W}=\sum_{k\neq l}r^{-2}=r/\left(  1-r\right)
$. Adding over $m$ we obtain
\begin{align*}
&  {\mathbb{E}}_{x}\left[  \left(  \sum_{j=1}^{r^{l_{n}\left(  x\right)  }%
}I_{j}\,{r^{-l_{n}\left(  x\right)  }}\right)  ^{2}\right] \\
&  =\mathcal{W}\sum_{m=0}^{l_{n}\left(  x\right)  -1}{\mathbb{E}}_{x}\left[
\sum_{j=1}^{r^{m}}I\left(  V\left(  Q_{m,1}\right)  >0\right)  r^{-2m}%
p_{n}^{V}\left(  Q_{m,j}\right)  ^{2}\right] \\
&  =\mathcal{W}\sum_{m=0}^{l_{n}\left(  x\right)  -1}r^{-m}{\mathbb{E}}%
_{x}\left[  I\left(  V\left(  Q_{m,1}\right)  >0\right)  p_{n}^{V}\left(
Q_{m,j}\right)  ^{2}\right]  .
\end{align*}
Combining this with the diagonal term in (\ref{eqn:expansion of 2nd mom}),
which can be readily expressed as $r^{-l_{n}\left(  x\right)  }p_{n}%
^{V}\left(  x\right)  $, we arrive at the following expansion for the second
moment of $R_{n}\left(  x\right)  $:
\begin{equation}
{\mathbb{E}}_{x}\left[  R_{n}\left(  x\right)  ^{2}\right]  =\mathcal{W}%
\sum_{m=0}^{l_{n}\left(  x\right)  -1}r^{-m}{\mathbb{E}}_{x}\left[  I\left(
V\left(  Q_{m,1}\right)  >0\right)  p_{n}^{V}\left(  Q_{m,1}\right)
^{2}\right]  +r^{-l_{n}\left(  x\right)  }p_{n}^{V}\left(  x\right)  .
\label{eqn:expand 2nd mom, fin}%
\end{equation}
The next result takes advantage of expression (\ref{eqn:expand 2nd mom, fin})
to obtain an upper bound for ${\mathbb{E}}_{x}\left[  R_{n}\left(  x\right)
^{2}\right]  $.

\begin{proposition}
\label{prop: 2nd mom upb} The second moment of $R_{n}\left(  x\right)  $
satisfies
\begin{equation}
{\mathbb{E}}\left[  R_{n}\left(  x\right)  \right]  ^{2}=p_{n}^{V}\left(
x\right)  ^{2}\,O\left(  n^{\beta}\right)  . \label{eqn:splitting 2nd mom}%
\end{equation}
where $\beta=\sum_{i=1}^{d}I\left(  \rho_{i}=\rho_{\ast}\right)  $ is the
number of bottleneck stations in the whole network.
\end{proposition}

In order to prove the previous result, we will show that the second moment of
$R_{n}\left(  x\right)  $ is dominated by the first item on the righthand side
of the equality in (\ref{eqn:expand 2nd mom, fin}). In turn, the asymptotic
behaviour of such term hinges on the conditional distribution of the exact
position of the particle in generation $m$, $Q_{m,1}$ in $C^{n}_{l_{n}\left(
x\right)  -m}$.

\begin{proof}
We begin the proof with an important property implied by the splitting
algorithm:
\begin{align}\label{eqn:splitting observation}
V\left(  Q_{m,1}\right)  >0  & \Leftrightarrow Q_{m,1}\in
C_{l_{n}\left(
x\right)  -m}^{n}=nL_{\left(  l_{n}\left(  x\right)  -m\right)  \Delta/n}\nonumber\\
& \Leftrightarrow Q_{m,1}\in\{z\in nD_{n}:U\left(  z/n\right)  \leq\left(
l_{n}\left(  x\right)  -m\right)  \Delta/n\}\nonumber\\
& \Leftrightarrow Q_{m,1}\in\{z\in nD_{n}:C\left(  \Delta-\frac{1}{n}%
\alpha^{T}z\Delta\right)  \leq\frac{\Delta}{n}\left(  \lceil C\left(
n-\alpha^{T}x\right)  \rceil-m\right)  \}\nonumber\\
& \Rightarrow Q_{m,1}\in\{z\in nD_{n}:C\left(  1-\frac{1}{n}\alpha
^{T}z\right)  \leq\frac{1}{n}\left(  C\left(  n-\alpha^{T}x\right)
-m+1\right)  \}\nonumber\\
& \Leftrightarrow Q_{m,1}\in\{z\in nD_{n}:\alpha^{T}z\geq\alpha^{T}%
x+\frac{m-1}{C}\}\nonumber\\
& \Leftrightarrow Q_{m,1}\in\{z\in
nD_{n}:\varrho^{T}z\leq\varrho^{T}x-\left( m-1\right)  \log r\}
\end{align}
where we used the representations of $U\left(  \cdot\right)  $ and
$l^{n}\left(  x\right)  $ in (\ref{eqn:importance function}) and
(\ref{eqn:level ind fun}) and the definition of $L_{z}$ in
(\ref{eqn:level fun}). In other words, if a particle survives $m$
generations then its current position is beyond the $m$th
level, which implies that the weighted sum of system population,
with weight given by the vector $\varrho$, is bounded from above by
that of the initial position adjusted by a
linear function in $m$. If we define the stopping time $\hat{T}_{\frac{m}{C}%
}\triangleq\inf\{k\geq1:\alpha^{T}Q\left(  k\right)  \geq\alpha^{T}%
x+\frac{m-1}{C}\}=\inf\{k\geq1:\varrho^{T}Q\left(  k\right)  \leq\varrho
^{T}x-\left(  m-1\right)  \log{r}\}$, the above property also implies that
$V\left(  Q_{m,1}\right)  >0\Rightarrow\hat{T}_{\frac{m}{C}}<T_{0}$. Now, the
expectation term in the sum of (\ref{eqn:expand 2nd mom, fin}) can be
expressed as
\begin{align}\label{eqn:2nd moment dom}
&  {\mathbb{E}}_{x}\left[  I\left(  V\left(  Q_{m,1}\right)
>0\right)
p_{n}^{V}\left(  Q_{m,1}\right)  ^{2}\right]  \nonumber\\
\leq &  {\mathbb{E}}_{x}\left[  I\left(  \varrho^{T}Q_{m,1}\leq\varrho
^{T}x-\left(  m-1\right)  \log{r}\right)  p_{n}^{V}\left(  Q_{m,1}\right)
^{2}\right]  \nonumber\\
= &  {\mathbb{E}}_{x}\left[  p_{n}^{V}\left(  Q_{m,1}\right)
^{2}|\varrho
^{T}Q_{m,1}\leq\varrho^{T}x-\left(  m-1\right)  \log{r}\right]  {\mathbb{P}%
}_{x}\left(  \hat{T}_{\frac{m}{C}}<T_{0}\right)  \nonumber\\
= &  {\mathbb{E}}_{x}\left[  {\mathbb{P}}_{x}^{2}\left(  T_{n}^{V}%
<T_{0}|Q_{\hat{T}_{\frac{m}{C}}}\right)  |\hat{T}_{\frac{m}{C}}<T_{0}%
,\varrho^{T}Q_{m,1}\leq\varrho^{T}x-\left(  m-1\right)
\log{r}\right] {\mathbb{P}}_{x}\left(
\hat{T}_{\frac{m}{C}}<T_{0}\right)
\end{align}
where we used the property derived in (\ref{eqn:splitting observation}). For
the first item in (\ref{eqn:2nd moment dom}), we have
\begin{align}
\label{eqn:2nd moment dom2}
&  {\mathbb{E}}_{x}\left[  {\mathbb{P}}_{x}^{2}\left(  T_{n}^{V}<T_{0}%
|Q_{\hat{T}_{\frac{m}{C}}}\right)  |\hat{T}_{\frac{m}{C}}<T_{0},\varrho
^{T}Q_{m,1}\leq\varrho^{T}x-\left(  m-1\right)  \log{r}\right]  \nonumber\\
\leq & K {\mathbb{E}}\left[  \frac{\pi^{2}\left(  C_{0}^{n}\right)
}{\pi ^{2}\left(  Q_{m,1}\right)
}|\varrho^{T}Q_{m,1}\leq\varrho^{T}x-\left(
m-1\right)  \log{r}\right]  \nonumber\\
\leq &  c_{1}\left[  n^{\beta_{V}-1}\left(  \rho_{\ast}^{V}\right)
^{n}\right]  ^{2}{\mathbb{E}}_{\pi}\left[  e^{-2\varrho^{T}Q_{m,1}}%
|\varrho^{T}Q_{m,1}\leq\varrho^{T}x-\left(  m-1\right)
\log{r}\right]
\end{align}
where $c_{1},K$ are some constants independent of $n$. Here we used Propositions 1
and 2 for the the last two inequalities respectively. As for the expectation
term in (\ref{eqn:2nd moment dom2}), since the process $Q\left(  \cdot\right)
$ has for each dimension an increment at most of unit size, we can write
\begin{align}\label{eqn:2nd moment dom cond}
& {\mathbb{E}}_{\pi}\left[  e^{-2\varrho^{T}Q_{m,1}}|\varrho^{T}Q_{m,1}%
\leq\varrho^{T}x-\left(  m-1\right)  \log{r}\right]  \nonumber\\
& ={\mathbb{E}}_{\pi}\left[  e^{-2\varrho^{T}Q_{m,1}}|\varrho^{T}x-\left(
m-1\right)  \log{r}-\delta\leq\varrho^{T}Q_{m,1}\leq\varrho^{T}x-\left(
m-1\right)  \log{r}\right]  \nonumber\\
& \leq c_{2}\exp\left(  -2\varrho^{T}x+2\left(  m-1\right)
\log{r}\right)
\nonumber\\
& =c_{3}\exp\left(  -2\frac{m-1}{C}\log\rho_{\ast}^{V}\right)
=c_{3}\left( \rho_{\ast}^{V}\right)  ^{-2\frac{m-1}{C}},
\end{align}
where $c_{2},c_{3}$ and $\delta$ are some positive constants. It remains to
give a bound on the term ${{\mathbb{P}}}_{x}\left(  \hat{T}_{\frac{m}{C}%
}<T_{0}\right)  $. As a result of Proposition 2,
\begin{align*}
& {\mathbb{P}}_{x}\left(  \hat{T}_{\frac{m}{C}}<T_{0}\right)  \\
& \leq\frac{1}{\pi\left(  x\right)  }{\mathbb{P}}\left[  \varrho^{T}Q\left(
\infty\right)  \leq\varrho^{T}x-\left(  m-1\right)  \log{r}\right]  \\
& =\frac{1}{\pi\left(  x\right)  }{\mathbb{P}}\left[  \hat{\alpha}^{T}Q\left(
\infty\right)  \geq\hat{\alpha}^{T}x+\frac{\left(  m-1\right)  }{\hat{C}%
}\right]
\end{align*}
where $\hat{C}=-\log\rho_{\ast}/\log{r}$, $\rho_{\ast}=\max_{i}\rho_{i}%
\in\left(  0,1\right)  $ and $\hat{\alpha}=\varrho\,/\rho_{\ast}=\left(
\log\rho_{1}/\log\rho_{\ast},...,\log\rho_{d}/\log\rho_{\ast}\right)  ^{T}$.
Note that $\hat{\alpha}_{i}\in\left(  0,1\right)  $. To finish the proof we
need the following Lemma.
\begin{lemma}%
\begin{align*}
{\mathbb{P}}\left[  \hat{\alpha}^{T}Q\left(  \infty\right)  \geq\hat{\alpha
}^{T}x+\frac{\left(  m-1\right)  }{\hat{C}}\right]    & =\Theta\left[\mathbb{P}\left(
Z\left(  \beta,1-\rho_{\ast}\right)  \geq\frac{m-1}{\hat{C}}\right)\right]  \\
& =\Theta\left[  \left(  \frac{m-1}{\hat{C}}\right)  ^{\beta-1}\left(
\rho_{\ast}\right)  ^{\frac{m-1}{\hat{C}}}\right]
\end{align*}
where $Z\left(  n,p\right)  $ denotes a $NBin\left(  n,p\right)  $ (negative
binomial) random variable.
\end{lemma}
\begin{proof}
Note that
\begin{align*}
\hat{\alpha}^{T}Q\left(  \infty\right)    & =Q\left(  \infty\right)  ^{T}%
\frac{\varrho}{\log\rho_{\ast}}\\
& =\sum_{i=1}^{d}Q_{i}\left(  \infty\right)  I\left(  \rho_{i}=\rho_{\ast
}\right)  +\sum_{i=1}^{d}Q_{i}\left(  \infty\right)  I\left(  \rho_{i}\neq
\rho_{\ast}\right)  \frac{\log\rho_{i}}{\log\rho_{\ast}}\\
& =Z\left(  \beta,1-\rho_{\ast}\right)  +W.
\end{align*}
One direction is elementary, since $\hat{\alpha}^{T}Q\left(  \infty\right)
\geq Z\left(  \beta,1-\rho_{\ast}\right)  $, we clearly have%
\begin{equation}
{\mathbb{P}}\left[  \hat{\alpha}^{T}Q\left(  \infty\right)  \geq\hat{\alpha
}^{T}x+\frac{\left(  m-1\right)  }{\hat{C}}\right]  \geq{\mathbb{P}}\left[
Z\left(  \beta,1-\rho_{\ast}\right)  \geq\hat{\alpha}^{T}x+\frac{\left(
m-1\right)  }{\hat{C}}\right]  .\label{lemma: lower bdd}%
\end{equation}
For the other direction, note that there exists a constant $\tilde{\rho}%
<\rho_{\ast}$ such that
\[
W=\sum_{i=1}^{d}Q_{i}\left(  \infty\right)  I\left(  \rho_{i}\neq\rho_{\ast
}\right)  \frac{\log\rho_{i}}{\log\rho_{\ast}}\leq\sum_{i=1}^{d}Q_{i}\left(
\infty\right)  I\left(  \rho_{i}\neq\rho_{\ast}\right)  =Z\left(
d-\beta,1-\tilde{\rho}\right)  .
\]
As a result,
\[
\hat{\alpha}^{T}Q\left(  \infty\right)  \leq_{st}Z\left(  \beta,1-\rho_{\ast
}\right)  +Z\left(  d-\beta,1-\tilde{\rho}\right)  ,
\]
where \textquotedblleft\ $\leq_{st}$\textquotedblright\ denotes that the left
hand side is stochastically dominated by the right hand side. But since
$1-\rho_{\ast}<1-\tilde{\rho}$, a similar argument as given by Proposition 3
in \cite{Blanchet09} allows us to obtain
\begin{equation}
{\mathbb{P}}\left[  \hat{\alpha}^{T}Q\left(  \infty\right)  \geq\hat{\alpha
}^{T}x+\frac{\left(  m-1\right)  }{\hat{C}}\right]  \leq c_{0}{\mathbb{P}%
}\left[  Z\left(  \beta,1-\rho_{\ast}\right)  \geq\hat{\alpha}^{T}%
x+\frac{\left(  m-1\right)  }{\hat{C}}\right]  ,\label{lemma: upper bdd}%
\end{equation}
for some finite constant $c_{0}$ that is independent of $m$. Combining
(\ref{lemma: lower bdd}) and (\ref{lemma: upper bdd}), we have
\begin{equation}
{\mathbb{P}}\left[  \hat{\alpha}^{T}Q\left(  \infty\right)  \geq\hat{\alpha
}^{T}x+\frac{\left(  m-1\right)  }{\hat{C}}\right]  =\Theta\left[
{\mathbb{P}}\left(  Z\left(  \beta,1-\rho_{\ast}\right)  \geq\hat{\alpha}%
^{T}x+\frac{\left(  m-1\right)  }{\hat{C}}\right)  \right]  .
\end{equation}
Using again Proposition 3 of \cite{Blanchet09}, we reach the conclusion that
\[
{\mathbb{P}}\left[  \hat{\alpha}^{T}Q\left(  \infty\right)  \geq\hat{\alpha
}^{T}x+\frac{\left(  m-1\right)  }{\hat{C}}\right]  =\Theta\left[  \left(
\frac{m-1}{\hat{C}}\right)  ^{\beta-1}\left(  \rho_{\ast}\right)  ^{\frac
{m-1}{\hat{C}}}\right]
\]
\end{proof}
\bigskip Going back to (\ref{eqn:2nd moment dom}), the above Lemma allows us
to write
\begin{equation}
{\mathbb{P}}_{x}\left(  \hat{T}_{\frac{m}{C}}<T_{0}\right)  \leq
c_{4}\left( \frac{m-1}{\hat{C}}\right)  ^{\beta-1}\left(
\rho_{\ast}\right)  ^{\frac
{m-1}{\hat{C}}}.\label{eqn:2nd moment dom 4}%
\end{equation}
If we combine (\ref{eqn:2nd moment dom2}), (\ref{eqn:2nd moment dom cond}) and
(\ref{eqn:2nd moment dom 4}), we obtain the following upper bound for the
expectation term in the sum of expression (\ref{eqn:expand 2nd mom, fin}):
\begin{align}\label{eqn:2nd moment dom 5}
{\mathbb{E}}_{x}\left[  I\left(  V\left(  Q_{m,1}\right)  >0\right)  p_{n}%
^{V}\left(  Q_{m,1}\right)  ^{2}\right]    & \leq c\,p_{n}^{V}\left(
x\right)  ^{2}\left(  \rho_{\ast}^{V}\right)  ^{-2\frac{m-1}{C}}\left(
\frac{m-1}{\hat{C}}\right)  ^{\beta-1}\left(  \rho_{\ast}\right)  ^{\frac
{m-1}{\hat{C}}}\nonumber\\
& =c\,p_{n}^{V}\left(  x\right)  ^{2}\left(  \rho_{\ast}^{V}\right)
^{-\frac{m-1}{C}}\left(  \frac{m-1}{\hat{C}}\right)  ^{\beta-1}\nonumber\\
& =c\,p_{n}^{V}\left(  x\right)  ^{2}r^{m-1}\left(  \frac{m-1}{\hat{C}%
}\right)  ^{\beta-1}%
\end{align}
where for the first equality we use the fact that $\rho_{\ast}^{V}\leq
\rho_{\ast}$ and $\hat{C}\leq C$, and for the second equality we use
$\rho_{\ast}^{V}=r^{-C}$. Putting the bound in (\ref{eqn:2nd moment dom 5})
back to the sum in the first item of (\ref{eqn:expand 2nd mom, fin}), we have
\begin{align}\label{eqn:first piece}
& \sum_{m=0}^{l_{n}\left(  x\right)  -1}r^{-m}{\mathbb{E}}_{x}\left[
I\left( V\left(  _{m,1}\right)  >0\right)  p_{n}^{V}\left(
_{m,1}\right)
^{2}\right]  \nonumber\\
& \leq cr^{-1}\sum_{m=0}^{l_{n}\left(  x\right)  -1}p_{n}^{V}\left(  x\right)
^{2}\left(  \frac{m-1}{\hat{C}}\right)  ^{\beta-1}\nonumber\\
& =p_{n}^{V}\left(  x\right)  ^{2}\,O\left(  n^{\beta}\right)
\end{align}
\indent Finally, note that the second item of (\ref{eqn:expand 2nd mom, fin})
is dominated by (\ref{eqn:first piece}), and it follows immediately that
\[
{\mathbb{E}}\left[  R_{n}\left(  x\right)  \right]  ^{2}=p_{n}^{V}\left(
x\right)  ^{2}O\left(  n^{\beta}\right)  .
\]
\smallskip
\end{proof}

Equipped with these results, we are ready to summarize our discussions in the
statement of the following Theorem, which is the main result of this paper.

\begin{theorem}
\label{thm: total complexity} To estimate the overflow probability $p_{n}%
^{V}\left(  x\right)  $ using $R_{n}\left(  x\right)  $, the number of
function evaluations needed for a given level of relative error is $O\left(
n^{\beta_{V}+\beta+1}\right)  $.
\end{theorem}

\begin{proof}
Recall from section \ref{complexity} that the number of function evaluations
sufficient to achieve a pre-determined level of relative accuracy for the
Splitting estimator is proportional to the work-normalized squared coefficient
of variation. This is therefore immediate by combining the upper bound
analysis of the computational effort per run in Proposition
\ref{prop: effort per run} along with the upper bound of the second moment of
$R_{n}\left( x\right) $ available in Proposition \ref{prop: 2nd mom upb}.
\end{proof}

A direct comparison to the $O\left(  n^{3d-2}\right)  $ complexity of solving
a system of linear equations (see Section \ref{complexity} ) yields the
immediate conclusion that the Splitting algorithm is \textquotedblleft
efficient\textquotedblright\ in the sense that it improves over the
\textquotedblleft benchmark\textquotedblright\ polynomial algorithm. Even in
the worst case when we look at the total population of the network and the
network is totally symmetric, i.e., all stations are bottlenecks $\left(
\beta_{V}=\beta=d>3\right)  $, the number of flops needed is off by a
substantial factor $n^{d-3}$. In the case where $\beta_{V}=\beta=1$, the
algorithm only requires a number of function evaluations that at most grows
cubically in the level of overflow $n$. Furthermore, if the number of
bottlenecks is less than half of the total number of stations, i.e.
$\beta<d/2$, the Splitting algorithm enjoys a running time of order smaller
than $O\left(  n^{d}\right)  $, which is not worst than storing the vector
that encodes the solution to the associated linear system. If, on the other
hand, more than half of the stations are bottlenecks, faster importance
sampling based algorithms do exist at least for the case of tandem networks;
see the analysis in \cite{BlanLedGlyn09}, which implies that $O\left(
n^{2(d-\beta)+1}\right)  $ function evaluations suffice to obtain an estimator
with a given relative precision. Overall, the analysis thus provides some sort
of guidance on the choice of simulation algorithms. It is meaningful to point
out that the previous comparison is not based on the sharpest analysis. In
fact we only resort to a rather crude upper bound in the analysis of the
second moment of $R_{n}\left(  x\right)  $ in (\ref{eqn:2nd moment dom2}). A
sharper result is possible by bounding the expectation term in
(\ref{eqn:rep_overflow}) with more care. But as pointed out in the
Introduction, even though there is still room for a more refined analysis, we
believe our work provides substantial insights leading to a better
understanding of the relations between these two classes of
algorithms.\newline


\bigskip



\bibliographystyle{plain}
\bibliography{mod_prob}

\end{document}